\documentclass[12pt]{amsart}
\usepackage{amsmath, amssymb, amsthm, latexsym}
\usepackage{spectralsequences}
\usepackage{mathrsfs}
\usepackage{amssymb}
\usepackage{latexsym}
\usepackage[center]{caption}
\usepackage[OT2,T1]{fontenc}
\DeclareSymbolFont{cyrletters}{OT2}{wncyr}{m}{n}
\DeclareMathSymbol{\Sha}{\mathalpha}{cyrletters}{"58}
\usepackage{tikz}
\newcounter{braid}
\newcounter{strands}
\pgfkeyssetvalue{/tikz/braid height}{1cm}
\pgfkeyssetvalue{/tikz/braid width}{1cm}
\pgfkeyssetvalue{/tikz/braid start}{(0,0)}
\pgfkeyssetvalue{/tikz/braid colour}{black}
\pgfkeys{/tikz/strands/.code={\setcounter{strands}{#1}}}

\makeatletter
\def\cross{%
  \@ifnextchar^{\message{Got sup}\cross@sup}{\cross@sub}}

\def\cross@sup^#1_#2{\render@cross{#2}{#1}}

\def\cross@sub_#1{\@ifnextchar^{\cross@@sub{#1}}{\render@cross{#1}{1}}}

\def\cross@@sub#1^#2{\render@cross{#1}{#2}}

\def\render@cross#1#2{
  \def\strand{#1}
  \def\crossing{#2}
  \pgfmathsetmacro{\cross@y}{-\value{braid}*\braid@h}
  \pgfmathtruncatemacro{\nextstrand}{#1+1}
  \foreach \thread in {1,...,\value{strands}}
  {
    \pgfmathsetmacro{\strand@x}{\thread * \braid@w}
    \ifnum\thread=\strand
    \pgfmathsetmacro{\over@x}{\strand * \braid@w + .5*(1 - \crossing) * \braid@w}
    \pgfmathsetmacro{\under@x}{\strand * \braid@w + .5*(1 + \crossing) * \braid@w}
    \draw[braid] \pgfkeysvalueof{/tikz/braid start} +(\under@x pt,\cross@y pt) to[out=-90,in=90] +(\over@x pt,\cross@y pt -\braid@h);
    \draw[braid] \pgfkeysvalueof{/tikz/braid start} +(\over@x pt,\cross@y pt) to[out=-90,in=90] +(\under@x pt,\cross@y pt -\braid@h);
    \else
    \ifnum\thread=\nextstrand
    \else
     \draw[braid] \pgfkeysvalueof{/tikz/braid start} ++(\strand@x pt,\cross@y pt) -- ++(0,-\braid@h);
    \fi
   \fi
  }
  \stepcounter{braid}
}

\tikzset{braid/.style={double=\pgfkeysvalueof{/tikz/braid colour},double distance=1pt,line width=2pt,white}}

\newcommand{\braid}[2][]{%
  \begingroup
  \pgfkeys{/tikz/strands=2}
  \tikzset{#1}
  \pgfkeysgetvalue{/tikz/braid width}{\braid@w}
  \pgfkeysgetvalue{/tikz/braid height}{\braid@h}
  \setcounter{braid}{0}
  \let\sigma=\cross
  #2
  \endgroup
}
\makeatother

\input xypic

\swapnumbers
\newtheorem{theorem}{Theorem}
\newtheorem{proposition}[subsection]{Proposition}
\newtheorem{lemma}[theorem]{Lemma}


\def\Z{\mathbb{Z}}

\def\R{\mathbb{R}}

\def\F{\mathbb{F}}

\def\Zpk{\mathbb{Z}/p^{k}}
\def\Zpk1{\mathbb{Z}/p^{k-1}}

\newcommand{\rref}[1]{(\ref{#1})}

\newcommand{\beg}[2]{\begin{equation}\label{#1}#2\end{equation}}
\def\r{\rightarrow}

\def\F{\mathbb{F}}

\def\sl2{\widetilde{SL_{2}(\Z)}}

\title[Equivariant operations in $THH$]{Equivariant operations in topological Hochschild homology}
\author{Po Hu, Igor Kriz, Petr Somberg and Foling Zou}
\thanks{The authors acknowledge support by grant GA\,CR 19-28628X.
Hu acknowledges the support of NSF grant DMS-2301520
Kriz also acknowledges the support of a Simons Foundation award 958219.}

\begin{document}
\maketitle

\begin{abstract}
We observe a new equivariant relationship between topological Hochschild homology and cohomology. We also 
calculate the topological Hochschild homology of the topological Hoch-schild cohomology of a finite
prime field, which can be viewed as a certain ring of structured operations in this case.
\end{abstract}

\section{Introduction}

Topological Hochschild homology, with its structure of a genuine $S^1$-equivariant spectrum, 
is a remarkably strong tool, which has surprising applications. Examples include the 
construction of topological cyclic cohomology $TC$, 
\cite{tc,hm,hm1,hm2,tc1}, which is an effective tool for computing algebraic K-theory of complete rings \cite{dm}. 
Constructed in the process, another form of topological cyclic cohomology, $TR$, was
found, which is interesting because of its close relationship with the De Rham-Witt complex (see e.g. \cite{hm1}).
It is also implicit
in the discussion of Bhatt, Morrow, and Scholze, \cite{bms,bms1} (see also \cite{ns,hksthh}) of one approach to prismatic cohomology, unifying several known cohomology theories in $p$-adic Hodge theory. After the present paper was submitted for publication,
$TR$ also occurred in other follow-up work, e.g. \cite{and1,and2}.

There is a remarkable asymmetry between topological Hochschild homology $THH(R)$
and topological Hochschild cohomology $THC(R)$, which is known to be an algebra
over the little $2$-cube operad (a structure originally conjectured by Deligne), but there is
no known counterpart of a genuine $S^1$-equivariant structure, which seems odd. 

\vspace{3mm}

There is an old suggestion that the topological Hochschild homology $THH(R)$ and topological Hochschild cohomology
$THC(R)$ for an associative $S$-algebra $R$ should somehow be dual (mentioned e.g.
in Schechtman's ICM talk \cite{schecht}). This stems 
from the idea of Koszul duality for operads \cite{gk}
which, even though its statements do not apply here literally (due to the fact that 
associative $S$-algebras do not form a based category), should still have some manifestation.

Yet the behaviors 
of the constructions of topological Hochschild homology and cohomology are quite different. As already remarked,
$THH(R)$ forms
a genuine $S^1$-spectrum.
This was first described by B\"{o}kstedt-Hsiang-Madsen \cite{tc} and used by Hesselholt-Madsen
\cite{hm,hm1,hm2}. 
In \cite{tc1} (see also \cite{bore}), a different approach using the multiplicative norm was given,
in terms of orthogonal spectra. In fact, the multiplicative norm can be defined
on the level of spaces, thus giving $THH(R)$ the structure of a genuine $S^1$-spectrum in the
formalism of $S$-modules as well. 

On the other hand, $THC(R)$ has a structure of a $\mathcal{C}_2$-algebra in the category of $S$-modules. 
This is the spectral form of the 
Deligne conjecture (see McClure and Smith \cite{mcclures} and Lurie \cite{c2} from the point
of view of quasicategories; one can also
interpret spectrally the proof for chain complexes given by Hu and Kaufmann \cite{hucacti,kauf}). How are these two structures dual, or, 
more broadly, what do they have in common? 

\vspace{3mm}

It is of course not uncommon in algebraic topology for one of the partners of homology and
cohomology to have more structure. For example, the ring-valued cohomology of a space is a ring,
while its homology is not. Still, when learning about the structures of algebraic
topology, we realize that homology of a space is a module over cohomology, a structure
which works well with duality, so the disparity between the homology and cohomology of a 
space is largely resolved. From this point of view, the asymmetry between $THH(R)$ and
$THC(R)$ for ring (or $A_\infty$ ring spectrum $R$) seems much more profound, with little 
work in this direction done up to this point.

The purpose of this paper is to give a step in this direction by considering topological Hochschild
homology of topological Hochschild cohomology:

\vspace{3mm}

\begin{theorem}\label{t1}
Let $R$ be an associative $S$-algebra.
There exists a natural $S^1$-equivariant twisted associative $S$-algebra
$THHC(R)_{S^1}$
(indexed over the complete
universe) equivalent to $THH(THC(R))$
over which there exists an $S^1$-equivariant twisted left module equivalent to $THH(R)$.
\end{theorem}

\vspace{3mm}
The concept of a twisted associative algebra and module embodies a dependence of the product
on the $S^1$-action, which is defined below in Subsection \ref{sstwisted}. Roughly
speaking, this is necessary due to the Gerstenhaber Lie bracket on
$THC(R)$. We shall further discuss this in Section \ref{s2} below (see the Comment under Lemma 
\ref{lmtilde}).
Formulating the requisite constructions requires solving numerous technical problems,
and none of the formalisms available in the literature seemed to fit our purposes completely.
For this reason, we actually develop new (or modified) approaches to parts of the $THH$ story 
from scratch, using $S$-modules \cite{ekmm}.

\vspace{3mm}

It is worth noting that considering topological Hochschild homology of $THC(R)$ is, in many ways, a novel direction. While $THH$ can be applied to any $A_\infty$ (i.e. coherently associative) ring, $THC(R)$ is not a type of ring one would usually think of in this context. It is typically highly non-connective (as we shall see), which makes, for example, many of the methods of Nikolaus and Scholze \cite{ns} not applicable. Because of this, it seemed to make
sense to do a calculation at least in one basic case. This is provided by the following result
(where for a $G$-equivariant spectrum $E$, with $G$ a finite group, $E^G$ denotes its $G$-fixed points, and for
a non-equivariant spectrum $X$, $X_*$ denotes its coefficients):

\vspace{3mm}

\begin{theorem}\label{t2}
We have
\beg{et21}{(THH(THC(H\F_p)))^{\Z/p^{r-1}}_*=F(H\Z/p^r,H\Z/p^r)_*[y]\otimes\Gamma_{\Z/p^r}(\rho)
}
where $\rho$ is in homological degree $-2$, $y$ is in homological degree $2$, and $\Gamma$ denotes
the divided power algebra. Additionally,
\beg{et22}{(TR(THC(H\F_p)))_*=(F(H\Z,H\Z)_*)^\wedge_p
\otimes\Gamma_{\Z}(\rho)\otimes \Lambda_\Z(q)
}
where $TR$ is the homotopy limit of $THH^{\Z/p^{r-1}}$ with respect to the map $R$ of \cite{hm} and
$q$ has homological degree $-1$.
\end{theorem}

\vspace{3mm}

Theorem \ref{t2} can also be interpreted as a calculation of a type of 
structured $THH$ operations, in the basic case of the perfect field $\mathbb{F}_p$. 
The answer is remarkably small, reminding us of the result of Caruso 
\cite{caruso} on 
the lack of $\Z/p$-equivariant cohomology operations. That, of course, was later explained
in \cite{hk,sw,hkszsteen}, 
where it was shown that to get all the expected operations, one needs
to consider a twist. This in fact suggests a connection between \cite{hkszsteen} and the present
paper, which is the question of the first k-invariant of $THH$, which will be pursued in 
subsequent work.

\vspace{3mm}

The present paper is organized as follows: Section \ref{sp} recalls some important preliminary
constructions, namely the multiplicative norm and the unframed cactus operad. In Section \ref{s1},
we reformulate the construction of $THH$ in a way which is compatible with the constructions
needed to prove Theorem \ref{t1}. Theorem \ref{t1} is proved in Section \ref{s2}.
Sections \ref{scal} and \ref{scalc1} serve to recall some preliminary material needed
in the proof of Theorem \ref{t2}. In Section \ref{scal}, we recall some facts about the
dual Steenrod algebra and about integral Steenrod operations. In Section \ref{scalc1}, 
we recall the calculation of the equivariant homotopy groups of
$THH(\F_p)$ from our present point of view. In Section \ref{str}, we prove Theorem \ref{t2}.

\vspace{5mm}

\noindent
{\bf Acknowledgement:} We are thankful to M.Mandell for comments.

\vspace{5mm}

\section{Preliminaries}\label{sp}

The main purpose of this section is to recall, and partially reformulate for the purposes of this
paper, two important preliminary constructions: the multiplicative norm and the unframed cactus operad.

\subsection{The multiplicative norm}

One of the subjects to address is the multiplicative norm of equivariant
$S$-modules. 
This was introduced by Hill, Hopkins, and Ravenel
in \cite{hhr} in the context of finite groups. (It had been previously introduced
by Hu \cite{hu} in the context of motivic
spectra, which was later studied in detail by Bachmann and Hoyois \cite{bh}.)
For the basics on equivariant spectra, we refer the reader to \cite{lms}.

Let $G$ be a compact Lie group. Suppose $H\subseteq G$ is a 
subgroup of finite index. Let $\mathcal{U}$ be a complete $H$-universe. Then we have a complete $G$-universe
$$\mathcal{V}=Ind_H^G(\mathcal{U}).$$
Enumerating, once and for all, the cosets in $G/H$, we obtain an inner product space isomorphism
\beg{eiso1}{\mathcal{V}\cong \bigoplus_{|G/H|}\mathcal{U}.}
(We use $g\in G/H$ to identify the $g$th copy of the $H$-universe $\mathcal{U}$ with a $gHg^{-1}$-universe.)
Now, for a $\mathcal{U}$-indexed Lewis-May $H$-spectrum $X$, \rref{eiso1} gives the external smash product
$$\underbrace{X\wedge\dots\wedge X}_{\text{$|G/H|$ copies}}$$
a structure of a $\mathcal{V}$-indexed $G$-spectrum. This construction becomes, in
an obvious way, a functor from $\mathcal{U}$-indexed $H$-spectra to $\mathcal{V}$-indexed 
$G$-spectra, which we denote by $N_H^G$.

Recall from \cite{ekmm} that
the construction extends to $\mathbb{L}$-spectra (and hence to $S$-modules),
and goes as follows: Recall that an
$\mathbb{L}$-spectrum is a spectrum $X$ with a map 
$$\mathcal{I}(\mathcal{U},\mathcal{U})\rtimes X\r X$$
satisfying the obvious associativity and unit properties. Now 
consider the coequalizer of 
\beg{eiso2}{\mathcal{I}(\mathcal{U}^n,\mathcal{V})\rtimes (\bigwedge_{|G/H|}\mathcal{I}(\mathcal{U},
\mathcal{U})\rtimes X)\rightrightarrows
\mathcal{I}(\mathcal{U}^n,\mathcal{V})\rtimes (\bigwedge_{|G/H|}X).}
The two arrows are defined by composing linear isometries, or by applying the action on $X$.
Now on \rref{eiso2}, both maps are actually morphisms of $G$-equivariant spectra indexed over $\mathcal{V}$,
if we use the $G$-action on $\mathcal{V}$, conjugation action on isometries and coset action on the 
smash components.

This defines a functor from $\mathcal{U}$-indexed $\mathbb{L}$-$H$-spectra to $\mathcal{V}$-indexed
$\mathbb{L}$-$G$-spectra, which further passes to $S$-modules. We denote all these functors by $N_H^G$.
By construction, we have a natural isomorphism
\beg{e1}{N_{\{e\}}^{\Z/n}(X)\wedge N_{\{e\}}^{\Z/n}(Y)\cong N_{\{e\}}^{\Z/n}(X\wedge Y)}
(where $\wedge$ denotes the symmetric monoidal smash product of $S$-modules),
satisfying the obvious associativity and commutativity properties.
The construction also preserves cell objects.

\vspace{3mm}

\subsection{The unframed cactus operad}

For our purposes, it is also appropriate to describe in detail the {\em unframed cactus operad} introduced
by Voronov \cite{voronov}. To start out, by a {\em cactus datum}, we shall mean a pair
$$(T,\mathcal{E})$$
where 
$$T=\{0=t_0<t_1<\dots<t_n=1\}$$
is a partition of the unit interval and $\mathcal{E}$ is an equivalence relation on $T$ such that

\vspace{2mm}
\begin{itemize}
\item
Every equivalence class of $\mathcal{E}$ has $>1$ element

\vspace{2mm}

\item
$0\sim 1$

\vspace{2mm}

\item
If $E_1=\{t_{i_1}<\dots <t_{i_k}\}$, $E_2=\{t_{j_1}<\dots<t_{j_\ell}\}$ are equivalence 
classes of $\mathcal{E}$, then one of the following occurs:

(a) There exist a $1\leq s<\ell$ such that
$$t_{j_s}<t_{i_1}<t_{i_k}<t_{j_{s+1}}$$
or

(b) There exist a $1\leq s<k$ such that
$$t_{i_s}<t_{j_1}<t_{j_\ell}<t_{i_{s+1}}.$$
or

(c) $t_{i_k}<t_{j_1}$

\noindent
or

(d) $t_{j_\ell}<t_{i_1}$.

\end{itemize}

\begin{figure}
\includegraphics{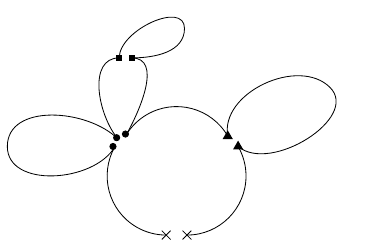}
\caption{An example of a cactus datum}
\end{figure}

\begin{figure}
\includegraphics{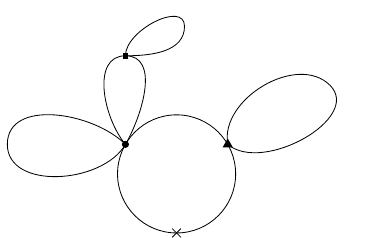}
\caption{The cactus graph corresponding to a cactus datum}
\end{figure}

\vspace{2mm}
The topology on the set of all cactus data is given by defining a sequence to converge
if it converges in the Hausdorff topology on the set
of equivalence classes of $\mathcal{E}$, and no two points in an equivalence class are identified in the limit.

\vspace{3mm}
The {\em cactus graph} $\Gamma(T,\mathcal{E})$ associated with a cactus datum $(T,\mathcal{E})$ is 
obtained by identifying the elements of $T$ which belong to the same equivalence class of $\mathcal{E}$.

A {\em cactus loop} of the graph $\Gamma(T,\mathcal{E})$ associated with the cactus datum $(T,\mathcal{E})$
is determined by choosing an $\mathcal{E}$-equivalence class
$$E=\{t_{i_1}<\dots<t_{i_k}\},$$
a number $1\leq s<k$, and taking the union of the images of all intervals $[t_j,t_{j+1}]$,
$i_s\leq j<i_{s+1}$
such that for all $i_s\leq p\leq j$, $j+1\leq q< i_{s+1}$,
$$t_p\nsim t_q.$$
One readily sees that a cactus loop is indeed a loop in the graph $\Gamma(T,\mathcal{E})$. Further,
form a $2$-dimensional CW-complex $\Gamma_2(T,\mathcal{E})$ by attaching a $2$-cell to each
cactus loop homeomorphically on the boundary. Further, choosing an orientation of the
$2$-cell so that its boundary intervals $[t_j,t_{j+1}]$
with increasing $j$ appear, say, in clockwise order, there is a unique (up to homeomorphism)
oriented embedding $\Gamma_2(T,\mathcal{E})\subset\R^2$.

Denote by $L(T,\mathcal{E})$ the set of loops of the cactus datum $(T,\mathcal{E})$.
One can check that the set of pairs
$$((T,\mathcal{E}),x),\;x\in L(T,\mathcal{E})$$
forms a covering space $\widetilde{X}$ over the space $X$ of all cactus data.

A {\em labelled cactus} consists of the data
$$(T,\mathcal{E}),\;
\diagram 
\sigma:L(T,\mathcal{E})\rto^\cong &\{1,\dots, N\}
\enddiagram
$$
where $(T,\mathcal{E})$ is a cactus datum. The {\em unframed cactus operad}
is the set of labelled cacti with topology induced from the covering space $\widetilde{X}$.

To define the operad structure, letting a loop $\ell\in L(T,\mathcal{E})$ consist of edges
\beg{eedges}{
[t_{j_1},t_{j_1+1}],\dots,[t_{j_m},t_{j_m+1}],
}
$j_1<j_2<\dots<j_m$, we have
\beg{eedges2}{
t_{j_s+1}\sim t_{j_{s+1}},
}
$s=1,\dots,m-1$. 

Identifying \rref{eedges2} in \rref{eedges}, we obtain an interval congruent (by an increasing map) to
$$J=[0,\displaystyle\sum_{s=1}^{m}(t_{j_s+1}-t_{j_s})]$$
where \rref{eedges2} goes to
$$\displaystyle
\sum_{p=1}^{s}(t_{j_p+1}-t_{j_p}).
$$
Let $h$ be the homothety mapping $J$ homeomorphically onto $[0,1]$.

Then a cactus datum 
$$(S,\mathcal{F})$$
is inserted into the loop $\ell$ by taking the partition
$$T\cup h^{-1}(S)$$
with the equivalence relation generated by $\mathcal{E}\cup h^{-1}(\mathcal{F})$. One checks
that, defining arity as the number of loops, this endows the space of all labelled cacti with the structure
of an operad $\mathcal{D}$.

It is a well-known fact, moreover, that
the resulting operad is equivalent to the little $2$-cube operad, even though a careful proof does not
seem to appear in the published literature. The proof consists of the following steps:

\noindent
{\bf Step 1:} The space of labeled cacti with $n$ loops is equivalent to the classifying space of
the pure braid group $\widetilde{B}_n$. To prove this, note that we have a fiber sequence
$$\bigvee_n S^1\r B\widetilde{B}_n\r B\widetilde{B}_{n-1}$$
where the projection is by dropping a braid. Accordingly, one sees using the standard method of
Dold-Thom \cite{dt} that the projection given by dropping the $n$-th loop from an arity $n$ labelled 
cactus to the space of arity $(n-1)$ labelled cacti is a quasifibration, and that the fiber has the
required homotopy type.

\vspace{2mm}

\noindent
{\bf Step 2:} Consider the ``corona cacti'' in 
$\mathcal{D}(n)$, i.e. labelled cacti which are bouquets of $n$ loops. These 
form a sub-operad. Similarly,
for the little $2$-cube operad $\mathcal{C}_2$, we have a suboperad $\mathcal{C}_1$ of little $1$-cubes.
Both of these suboperads are manifestly equivalent to the associative operad, whose $n$-th term is
the symmetric group $\Sigma_n$.
Choose base-points of $\mathcal{C}_2(n)$ and $\mathcal{D}(n)$ in the corresponding suboperad, corresponding to the
same permutation. Then denoting by $\widetilde{X}$ the universal cover
of a based space $X$, the sequence $\mathcal{P}$ of spaces $\widetilde{\mathcal{D}(n)}\times_{\widetilde{B}_n}
\widetilde{\mathcal{C}_2(n)}$ (formed from the product of universal covers by
factoring using the diagonal action of $\widetilde{B}_n$) forms an operad. (Caution: it is {\em not} true that the universal cover
of either $\mathcal{D}$ or $\mathcal{C}_2$ forms an operad: this would give a map of operads from
an $E_\infty$-operad to the little $2$-cube operad, which is easily seen not to exist by Dyer-Lashof operations.
However, our base point choice is consistent on both coordinates, and thus, taking the quotient with respect to 
the diagonal action of $\widetilde{B}_n$ on both universal covers gives a consistent formula for the multiplication
regardless of the choice of base point, since this is true in each coordinate.)

\vspace{2mm}

\noindent
{\bf Step 3:} We have a diagram of equivalences
$$\mathcal{D}\leftarrow \mathcal{P}\rightarrow\mathcal{C}_2.$$

\begin{figure}
\resizebox{0.9\hsize}{!}{
\includegraphics{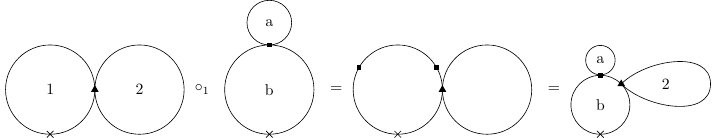}}
\caption{Cactus composition: Example 1}
\end{figure}

\begin{figure}
\resizebox{0.9\hsize}{!}{
\includegraphics{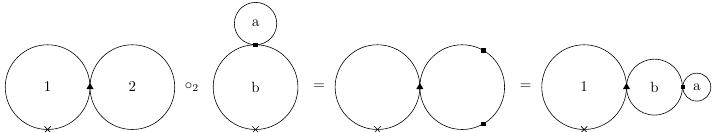}}
\caption{Cactus composition: Example 2}
\end{figure}

\subsection{Twisted equivariant algebras and modules}\label{sstwisted}

When discussing an $S^1$-equivariant associative $S$-algebra, we ordinarily mean that the
product is independent of the $S^1$-action. We shall need a more general concept,
where the product is allowed to change with the action. One possible way to formalize this
is as follows. We already encountered the associative operad $\mathcal{M}$ whose $n$-th
term is the symmetric group on $n$ elements:
$$\mathcal{M}(n)=\Sigma_n.$$
Now let $\widetilde{\mathcal{M}}$ be the free operad on the sequence
$$\Sigma_n\times S^1,$$
modulo the relation that for $x\in \Sigma_n$, $y_i\in \Sigma_{k_i}$, $i=1,\dots,n$,
$t\in S^1$, 
$$\gamma((x,t);(y_1,t),\dots,(y_n,t))\sim (\gamma(x;y_1,\dots,y_n),t),$$
$$(1,t)\sim (1,t^\prime), \; (0,t)\sim (0,t^\prime)$$
where $\gamma$ denotes operad composition and $1$ 
resp. $0$ is the single element of $\Sigma_1$ resp. $\Sigma_0$.

Then $\widetilde{\mathcal{M}}$ is an $S^1$-equivariant operad. A based $\widetilde{\mathcal{M}}$-algebra
will be called an $S^1$-equivariant {\em twisted associative algebra}.  Along with an $S^1$-equivariant twisted
associative algebra $A$,
we will also need a concept of an $S^1$-equivariant twisted left $A$-module. To this end, we recall that
operad modules over the associative operad model bimodules. The appropriate way of defining 
left and right modules over the associative operad and its variants is treated in \cite{drt}. Essentially, one fixes the
module coordinate as the first or the last one, thus restricting the permutations one is allowed to apply.
This treatment applies without change to the $S^1$-equivariant operad $\widetilde{M}$. 

\vspace{3mm}

\vspace{5mm}

\section{A description of $THH(R)$}\label{s1}

By a {\em cyclically ordered finite set}, we mean a finite set embedded to $S^1$. Two embeddings are considered
the same when they are related by an orientation-preserving diffeomorphism of $S^1$. A {\em cyclically ordered set}
is a set $Q$ each of whose finite subsets are cyclically ordered in a fashion compatible with inclusion.
A cyclic ordering on a set $Q$ is determined by a ternary relation of being ``in the anti-clockwise order."
A {\em morphisms of cyclically ordered sets} $f:Q\r Q^\prime$ is a map of sets where whenever $x,y,z$ are in the
anticlockwise order in $Q$, then either $f(x),f(y), f(z)$ are in the anticlockwise order, or $|\{f(x),f(y),f(z)\}|<3$.
The category of finite cyclically ordered sets will be denoted by $\Phi$.

Denoting by $Top$ the category of topological spaces, we have a canonical functor 
$$T:\Phi^{Op}\r Top$$
where $T(Q)$ is the space of morphisms of cyclic sets
$$Q\r S^1$$
(with the subspace topology of $(S^1)^Q$).

Let $R$ be a cell associative $S$-algebra. Then there is a natural functor
$TH^R$ from $\Phi$ into $S$-modules given by setting
$$TH^R(Q)=R^{\wedge Q}$$
where the action of surjective (resp. injective) morphisms of cyclically ordered finite sets is by multiplication in $R$
(resp. by insertion of units). Here $\wedge$ denotes the symmetric monoidal smash product of $S$-modules.

We consider the coend
$$THH(R)=T_+\wedge_{\Phi}TH^R(Q).$$
To give $THH(R)$ a genuine $\Z/n$-equivariant structure, we denote by $\Phi_n$ the category whose objects
are $\Z/n$-equivariant cyclically ordered sets $Q$ such that $\Z/n$ acts freely on $Q$ in a way which preserves
cyclic ordering and applying the $\Z/n$-action to any element 
$x\in Q$ defines a morphism of cyclically ordered sets
$$\Z/n\r Q.$$
Morphisms in $\Phi_n$ are morphisms of cyclically ordered sets which preserve the $\Z/n$-action.
We have, again, a canonical functor
$$T_n:\Phi^{Op}_n\r \Z/n\text{-}Top$$
where $T_n(Q)$ is the space of morphisms of $\Z/n$-equivariant cyclically ordered sets
$$Q\r S^1$$
(where we consider the standard $\Z/n$-action on $S^1$). Denoting for $Q\in Obj(\Phi_n)$ 
by $Q/(\Z/n)$ its set of orbits, we also have a functor from $\Phi_n$ to genuine $\Z/n$-equivariant
spectra (i.e. indexed by the complete universe)
$$TH_n^R(Q)=(N_{\{e\}}^{\Z/n}R)^{\wedge Q/(\Z/n)}.$$
(This uses \rref{e1}.) We set
$$THH(R)_{\Z/n}=(T_n)_+\wedge_{\Phi_n}TH_n^R.$$

By \rref{e1}, for $\Z/m\subset \Z/n$, we have a natural morphism 
$$res_{\Z/m}^{\Z/n} THH(R)_{\Z/n}\r THH(R)_{\Z/m}$$
which is an isomorphism. (This last property uses non-trivialy the unitality of $R$.) 

Thus, 
we are in the situation of a naive $S^1$-equivariant spectrum $E$
compatible with a system of genuine $\Z/n$-equivariant spectra $E_{\Z/n}$
for all natural numbers $n$, with isomorphisms
\beg{eres100}{res^{\Z/n}_{\Z/m}E_{\Z/n}\cong E_{\Z/m}.}
(We note that if we had a genuine $S^1$-equivariant spectrum, we woud obtain the data \rref{eres100}.)
This functor $res^{\Z/n}_{\Z/m}$ has both a left and a right adjoint, the left and right Kan extension,
respectively. (In fact, the Wirthm\"{u}ller isomorphism
states that these adjoints are naturally isomorphic in the derived category of equivariant spectra, see \cite{lms}, Section II.6.)
In the present setting, we apply the right adjoint to obtain a genuine $S^1$-equivariant spectrum $THH(R)_{S^1}$.
This functor can be loosely described as $F(E\mathcal{F}[S^1]_+,E)$ (where $\mathcal{F}[H]$ denotes
the family of subgroups not containing $H$). The reason why we prefer the right to the left adjoint
is that it preserves algebraic structure when present (for example $E_\infty$-structure for an $E_\infty$-ring spectrum
$R$).
Similar considerations in the settings of finite groups are discussed in \cite{hksf}.

The delicate point of this construction from the point
of view of isotropy separation is the piece $F(\widetilde{E\mathcal{F}[S^1]},?)$ (where
$\widetilde{X}$ denotes the unreduced suspension of $X$). It is worth noting
that this issue is present in some form in all known constructions of $THH$, and we do not know of
any setting where this particular piece is given any geometric interpretation in the case of $THH$.

\vspace{5mm}

\section{The action of $THH(THC(R))$}\label{s2}

In this section, we construct the $S^1$-equivariant $S$-algebra 
$$THHC(R)_{S^1}$$ 
of Theorem \ref{t1}, 
and show that it is equivalent to $THH(THC(R))$.
Now let $\Xi$ be the category whose objects are finite ordered sets and morphisms are non-decreasing 
maps. We have a functor
$$J:\Xi^{Op}\r Top$$
where $J(Q)$ is the set of non-decreasing maps
$$Q\r [0,1]$$
(with the subspace topology of $[0,1]^Q$). 

For an associative $S$-algebra $R$, we also have a functor $TJ^R$ from $\Xi$ to
$S$-modules given where
\beg{e3e1}{TJ^R(Q)=R^{\wedge Q}.}
(Again, surjective morphisms act by multiplication in $R$, while injective morphisms act by
insertion of units.)

We let
\beg{e3e0}{TI(R)=J_+\wedge_{\Xi}TJ^R.}
We note that $TI(R)$ is an $R^{Op}\wedge R$-module by multiplying at $0$ resp. $1$
from the left resp. right.

One notices that a pair of strictly monotone onto maps $[0,1]\r [0,s]$, $[0,1]\r [s,1]$
induces an isomorphism
\beg{e2}{TI(R)\wedge_R TI(R)\cong TI(R).
}
Now we may define
\beg{e3}{THC(R)=F_{R^{Op}\wedge R}(TI(R),R).
}
By \rref{e2}, we obtain an action of the unframed cactus operad (see Section \ref{sp} above) on $THC(R)$.
To explain
this action, we need to
explain how a cactus datum $(T,\mathcal{E})$ with $n$ loops maps from a smash of copies
of \rref{e3} for each loop to a single copy of \rref{e3}.

We invite the reader to look at Figure 1. We see that the cactus operad is generated by $\mathcal{D}(2)$,
so it suffices to describe the composition in the case of a cactus with two loops. (This is not really necessary,
but the general case is more complicated, so reducing to two loops may make the description easier to
understand.) 
In the case of two loops, the loops may either be attached both to the base point (call this Case 1) or one loop
could be attached to a non-boundary point of another (call this Case 2). 

In Case 1, the cactus datum gives a partition of the unit interval $[0,1]$ into two subintervals $[0,s]$ and
$[s,1]$, $0<s<1$. To describe a morphism
\beg{e3333}{TI(R)\wedge F_{R^{Op}\wedge R}(TI(R),R)\wedge F_{R^{Op}\wedge R}(TI(R),R)\r R,}
first break up $TI(R)$ according to our partition using \rref{e2}, then apply the evaluation map
on each factor paired with the corresponding copy of $F_{R^{Op}\wedge R}(TI(R),R)$,
then multiply.

In Case 2, the intervall $[0,1]$ is being partitioned into three subintervals $[0,s]$, $[s,t]$, $[t,1]$,
$0<s<t<1$. The interval $[s,t]$ corresponds to the loop being inserted, the remaining two intervals
correspond to the loop we are insterting into. To describe \rref{e3333} in this case, break up 
$TI(R)$ into three copies of $TI(R)$ according to our partition, using \rref{e2} twice. Apply evaluation of the middle copy
with respect to the loop we are inserting first, giving a copy of $R$. Using the inverse of \rref{e2} on the
intervals $[0,s]$, $[t,1]$ together with 
the $R^{Op}\wedge R$-module structure to multiply by the additional copy of $R$ in the middle, we get 
io a single copy of $TI(R)$. Pair it with the evaluation map of the loop being inserted into to get into $R$.

It is now clear that the general case can be described similarly, or, alternatively, one could approach this
by verifying standard relations on three loops. We skip this.

In particular,
the data giving \rref{e2} thereby induce a morphism
\beg{e4}{THC(R)\wedge THC(R)\r THC(R).}
We will now replace $THC(R)$ with its cell approximation in the category of $S$-algebras over the unframed
cactus operad. For simplicity, from now on, we suppress the approximation from the notation.

The map \rref{e4} is defined up to homotopy, but we can rigidify it by considering the operad $\mathcal{D}_1$
where $\mathcal{D}_1(n)$ consists of partitions of the interval $[0,1]$ into $n$ subintervals, which are labelled
by numbers $1,\dots,n$. Then \rref{e2} defines an action of the operad $\mathcal{D}_1$ on $THC(R)$.
The operad $\mathcal{D}_1$ is equivalent to the associative operad, so a $\mathcal{D}_1$-algebra
can be rectified into an associative algebra e.g. by bar construction of monads.

\vspace{3mm}
To consider $THH(THC(R))$, however, rectification is awkward, so it is convenient to describe
a variant of our construction for a $\mathcal{D}_1$-algebra $\mathcal{R}$.
Consider the space $TJH$ of finite sets $Q$ of closed subintervals of $S^1$, none of which
are a single point, with disjoint
interiors, whose union has a non-empty complement. 
Then $Q$ is a cyclically ordered set. Letting $\Phi^0$ be
the category of finite cyclically ordered sets and isomorphisms, then we may consider
the contravariant functor 
$$TJH^0:\Phi^0\r Top$$
given by
sending $Q$ to the space of isomorphims of cyclically ordered
sets from $Q$ to an element of $TJH$. Then we may define
the spectrum
$$TH\mathcal{H}^0(\mathcal{R})=(TJH^0)_+\wedge_{\Phi^0} TH^{\mathcal{R}}$$
where $TH^{\mathcal{R}}$ is defined by the same formula as in Section \ref{s1}.
Additionally, using \rref{e4}, we may construct from $TH\mathcal{H}^0(\mathcal{R})$ an $S$-module 
$TH\mathcal{H}(\mathcal{R})$
by imposing a colimit identification where configurations in $TJH^0$ which contain a pair of intervals
sharing a boundary point are identified with the configuration where the two intervals are merged.

More precisely, let $\Phi^1$ denote the set of cyclically ordered sets with a distinguished
element. $TJH^1$ denote the set of all isomorphism of cyclically ordered sets
$Q\in \Phi^1$ with an element of $TJH$
where the end point of the distinguished
interval $J$ is equal to the beginning point of the next interval $J^\prime$, counted counter-clockwise. 
In addition to the forgetful map $TJH^1\r TJH^0$, we then have another map
$\phi^1:TJH^1\r TJH^0$ given by replacing the intervals $J$, $J^\prime$ with their union. Then consider
\beg{eaux}{(TJH^1)_+\wedge_{\Phi^1}TH^{\mathcal{R}}}
Then there are two morphisms
from \rref{eaux} to $TH\mathcal{H}^0(\mathcal{R})$, one given by applying inclusion to the first coordinate, the other
by applying $\phi_1$ to the first coordinate and the map \rref{e4} to the second. We define
$TH\mathcal{H}(\mathcal{R})$ as the coequalizer of these two morphisms.

We shall specifically put
\beg{ethhhc}{
THHC(R)=TH\mathcal{H}(THC(R)).
}

\begin{lemma}\label{lmtilde}
$THHC(R)$ has a structure of an $S^1$-equivariant twisted associative algebra, and $THH(R)$ is its 
$S^1$-equivariant twisted left module.
\end{lemma}

\begin{proof}
Both the composition and the module action are defined by contracting each of the intervals occurring in the definition of $TJH^0$ to a point
in the $THHC(R)$ construction. 
This creates a model of $S^1$, which can be canonically identified with $S^1$ (by a similitude) provided
we are given a base-point. Further, the composition is associative when the base-points match. (Merging
intervals is handled by \rref{e2}.) 

The base-point
moves by the $S^1$-equivariant structure, which is where the twisting comes from. 

Unitality also must be treated carefully. (Again, the discussion is parallel for the algebra
and module structures.) The point is that the cactus operad, as defined, is not based. It can be given a based structure
by consider an ``empty cactus,'' i.e. a single point. Then $THC(R)$ has the structure of a unital algebra
where the unit is given by multiplication 
\beg{e3e10}{R\wedge\dots\wedge R\r R.}
To see this, recall the definition \rref{e3}. To produce the unit
$$S\r F_{R^{Op}\wedge R}(TI(R),R),$$
we need to, by adjunction, produce a morphism of $R$-bimodules
$$TI(R)\r R.$$
To this end, apply the projection $J_+\r S^0$ in \rref{e3e0}, and then
apply \rref{e3e10}, using \rref{e3e1}.

To have unitality, then, the definition of $THHC(R)$ needs to be further modified by allowing an 
empty configuration of intervals, and allowing an interval to which the unit is plugged in to degenerate
to length $0$ and disappear. Constructions of this type are standard (they were repeatedly
applied, for example, in \cite{ekmm}). We omit the details.
\end{proof}

\vspace{3mm}
\noindent
{\bf Comment:} One easily sees examples where, due to the scaling after omitting intervals, 
moving the base point can change the order of multiplication of loops. Moving loops past each other
is a cactus-level interpretation of the Gerstenhaber bracket. For this reason, the twisting is non-trivial
in examples where the Gerstenhaber product on $THC(R)$ is non-trivial. The simplest example is the 
suspension spectrum $R$ on the free group $F(a_1,\dots,a_n)$ on $n$ generators with a disjoint
base point, in which case (omitting suspension spectra notation), 
$THC(R)$ is the homotopy equalizer of the two maps
$$F(a_1,\dots,a_n)_+\rightrightarrows\prod_{i=1}^{n} F(a_1,\dots, a_n)_+\{da_i\}$$
where the two maps send a word $w$ to $wa_ida_i$, $a_iwda_i$, respectively.

\vspace{3mm}
To complete the picture, we must, therefore, 
compare the $THH$ and $TH\mathcal{H}$-constructions for associative algebras.

\vspace{3mm}
\begin{lemma}\label{lthhc}
Let $R$ be an associative $S$-algebra. Then we have an equivalence
$$THH(R)\sim TH\mathcal{H}(\mathcal{R}).$$
\end{lemma}

\begin{proof}
We shall construct a spectrum $TH\Gamma(R)$ and equivalences
\beg{ethhccomp}{TH\mathcal{H}(R)\r TH\Gamma(R)\leftarrow THH(R).}
The idea is to replace $TJH$ with an analogue $TKH$ which allows the case where intervals
are a single point. However, one must address the issue that single points are allowed to map to the
same point while the non-degenerate intervals are not. One possible approach is to define $TKH$
as the set of maps of a cyclically ordered set to the space consisting of closed subintervals of $S^1$
(where single point intervals are allowed) where elements $a<b<c$ map into intervals $I,J,K$ which have
disjoint interiors (where the interior of a single point interval is defined to be empty), and
for all $x\in I$, $y\in J$, $z\in K$, we have $x\leq y\leq z$. As above, one can define
$$TH\Gamma^0(R)=(TKH^0)_+\wedge_{\Phi^0} TH^{R}$$
and then identify further through merging consecutive intervals in the cyclic order (containing intervals
consisting of a single point), using the multiplication $R\wedge R\r R$.

We then have obvious comparison maps \rref{ethhccomp}. To prove that the maps are equivalences,
introduce an increasing filtration by the number $n$ of elements of the set $Q$, and note that passing
from $n$ to $n+1$ consists of taking a homotopy pushout with the comparison maps preserving the pushout diagram,
while the morphisms on its corners are equivalences.
\end{proof}

\vspace{5mm}

To obtain a $\Z/n$-equivariant version, we consider similarly 
the space $TJH_n$ of collections in $TJH$ which are invariant under the standard $\Z/n$-action. We have
a category $\Phi^0_n$ with the objects $Obj(\Phi_n)$ and morphisms the isomorphisms in $\Phi_n$. We define
$TJH_n^0$ as the category of all isomorphisms from an object of $\Phi^0_n$ to an element of $TJH_n$.
We can then define
$$THHC^0(R)_{\Z/n}=(TJH_n)_+\wedge_{\Phi^0_n} TH^{THC(R)}_n.$$
Again, identifications can be imposed when a $\Z/n$-invariant $n$-tuple of pairs of intervals sharing a boundary
point is present.

More precisely, we define $\Phi^1_n$ as the set of pairs
$(Q,x)$ where $Q\in Obj(\Phi_n)=Obj(\Phi^0_n)$ and $x\in Q$.
Then let $TJH^1_n$ be the set of $\Phi^0_n$-morphisms from $Q$ with $(Q,x)\in \Phi^1_n$
to an element of $TJH_n$, where, again, 
the end point of the distinguished
interval $J$ is equal to the beginning point of the next interval $J^\prime$, counted counter-clockwise. 
In addition to the forgetful map $TJH^1_n\r TJH^0_n$, we then, again, have a map $\phi^1_n:
TJH^1_n\r TJH^0_n$ given by replacing $J$ and $J^\prime$ with their union, and similarly for all
the $\Z/n$ images of the pair $J,J^\prime$. Again, we then have two morphism from 
$$(TJH^1_n)_+\wedge_{\Phi^1_n}TH^{THC(R)}_n$$
to $THHC^0_n(R)$ given by the forgetful map and by applying $\phi^1_n$ in the first coordinate
and \rref{e1}, \rref{e4} in the second. We denote by $THHC_n(R)$ the coequalizer of these
morphisms.

By construction, then, we thus obtain a genuine $\Z/n$-equivariant associative $S$-algebra $THHC_n(R)$
which
acts $\Z/n$-equivariantly on 
$$THH_n(R).$$ 
Again, this data is compatible under restriction, thus 
creating a genuine $S^1$-equivariant $S$-algebra $THHC(R)_{S^1}$ acting on $THH(R)_{S^1}$.
This completes our proof of the last statement of Theorem \ref{t1}.

\vspace{5mm}

\section{Some recollections on the Steenrod algebra}\label{scal}

In this section, we will recall some facts about cohomological operations which will be needed in the next section.
We refer the reader to Kochman \cite{koch} as a general reference. First of all, we recall that on
$$A^*_p=F(H\Z/p,H\Z/p)_*,$$
multiplication by the Bockstein $Q_0$ from the left (or the right) is exact in the sense that its kernel is equal
to its image. These two maps are induced by maps
$$Q_0^L,Q_0^R:F(H\Z/p,H\Z/p)\r\Sigma F(H\Z/p,H\Z/p),$$
given by applying the Bockstein map in the first resp. the second coordinate. Thus, the Bockstein spectral sequence
for $F(H\Z,H\Z/p)_*$ collapses to $E_2=0$, and we have
$$0=p:F(H\Z,H\Z/p)\r  F(H\Z,H\Z/p),$$
since it is $0$ on coefficients and both the source and the target are generalized Eilenberg-MacLane spectra. The 
cofibration sequence
$$\diagram H\Z\rto^p &H\Z\rto& H\Z/p\enddiagram$$
then gives a splitting
$$F(H\Z/p,H\Z/p)=F(H\Z,H\Z/p)\vee \Sigma^{-1}F(H\Z,H\Z/p).$$
Now the $E_1$ of the Bockstein spectral sequence from $F(H\Z,H\Z/p)_*$ to $F(H\Z,H\Z)_*$ 
is a Koszul complex (by Milnor's relation \cite{milnor}, Theorem 4a). 
On coefficients, the image of the Bockstein on the second
coordinate of $F(H\Z,H\Z/p)$ is therefore equal to the coefficients of some generalized Eilenberg-MacLane spectrum, 
which we will denote by $P$. We therefore conclude that
\beg{est1}{F(H\Z,H\Z)=H\Z\vee P.}
We may now work backward, studying the effect of $p^k$ on either coordinate, finding for example that
\beg{est2}{F(H\Z,H\Z/p^k)=H\Z/p^k\vee P\vee \Sigma P.}
Therefore, $p^k$ is $0$ on \rref{est2}, and we obtain, in general,
\beg{est3}{F(H\Z/p^k,H\Z/p^k)=F(H\Z,H\Z/p^k)\vee \Sigma^{-1}F(H\Z,H\Z/p^k),
}
where the splitting is induced by the cofibration sequence
$$\diagram H\Z\rto^{p^k} &H\Z\rto& H\Z/p^k\enddiagram$$
in the first coordinate. Symmetric statements also hold in the other coordinate.

\begin{lemma}\label{lst}
The fiber of the morphism 
\beg{est5}{F(H\Z/p^k,H\Z/p^k)\r \Sigma F(H\Z/p^\ell,H\Z/p^\ell)}
given by the difference of the Bockstein maps in both summands \rref{est3} is
canonically (in the stable homotopy category) equivalent to
$$F(H\Z/p^{k+\ell}, H\Z/p^{k+\ell}).$$
\end{lemma}
\begin{proof}
A direct consequence of \rref{est3}
\end{proof}

It is also useful to realize directly what happens on coefficients in terms of
\rref{est2}. We may write \rref{est5} on coefficients as
\beg{est6}{\diagram
\Z/p^k_0\oplus P^*_0\oplus P^*[1]_0\oplus \Z/p^k[-1]_{-1}\oplus P^*[-1]_{-1}\oplus P^*_{-1}\dto\\
\Z/p^\ell[1]_0\oplus P^*[1]_0\oplus P^*[2]_0\oplus \Z/p^\ell_{-1}\oplus P^*_{-1}\oplus P^*[1]_{-1}
\enddiagram
}
(Here the subscript $0$ resp. $-1$ denotes a part of the coefficients of the term \rref{est2} 
which is unsuspended resp. suspended by $-1$.)

Now one of the Bocksteins sends $P^*[1]$ isomorphically to $P^*[1]_0$, while the other sends it
isomorphically to $P^*[1]_{-1}$. On the other hand, one of the Bocksteins sends $P^*_0$ to $P^*_{-1}$,
while the other sends $P^*_{-1}$ isomorphically to $P^*_{-1}$. The surviving terms give the answer, 
with the surviving target terms desuspended by $1$. Extensions are present between the copies of
$\Z/p^k$ and $\Z/p^\ell$, due to the definition of the Bockstein. Bookkeeping completes the result.

\vspace{3mm}

Also, using the splitting \rref{est3}, we obtain a canonical (up to homotopy) map
$$ F(H\Z/p^k,H\Z/p^k)\r F(H\Z/p^{k-1},H\Z/p^{k-1})
$$ 
and we have
\beg{est10}{\begin{array}{l}\displaystyle
\operatornamewithlimits{holim}_{k}(\dots\r F(H\Z/p^k,H\Z/p^k)\r F(H\Z/p^{k-1},H\Z/p^{k-1})\r\dots)\\[3ex]
=F(H\Z,H\Z)^\wedge_p\vee \Sigma^{-1} F(H\Z,H\Z)^\wedge_p.
\end{array}}

\vspace{5mm}

\section{A recollection of $THH(H\F_p)$}\label{scalc1}

We begin with recalling the calculation of $THH(H\F_p)^{\Z/p^{r-1}}$ of \cite{hm}, 4.2 - Theorem 4.5. One has
\beg{ethh1}{THH(H\F_p)=B_{H\F_p}(H\F_p,HF_p\wedge H\F_p,H\F_p),
}
(where the bar construction is in the category of $H\F_p$-algebras and everywhere we
assume cofibrant models). The coefficients of \rref{ethh1} can be calculated 
via the Eilenberg-MacLane spectral sequence which collapses to $E^2$ for $p=2$
and has a Kudo differential for $p>2$. In both cases, we can conclude that
\beg{ethh2}{THH(H\F_p)_*=\Z/p[\sigma]
}
where $\sigma$ is in homological degree $2$. Now 
the Tate spectral sequence for $THH(H\F_p)$ is
\beg{ethh3}{\Z/p[\sigma][x,x^{-1}]\otimes\Lambda_{\F_p}(u)\Rightarrow \widehat{THH(H\F_p)}^{\Z/p^{r-1}}
}
where $x$ is the Tate periodicity element of homological degree $-2$ and $u$ has homological degree $-1$.
Then \cite{hm}, 4.3-4.4 prove that $x$ is a permanent cycle, while
\beg{ethh4}{d^{2r-1}u=x^{r}\sigma^{r-1}.
}
Furthermore, there is a multiplicative extension
\beg{ethh5}{p=x\sigma,
}
therefore giving
\beg{ethh6}{\widehat{THH(H\F_p)}^{\Z/p^{r-1}}_*=\Z/p^{r-1}[x,x^{-1}].
}
The $\Z/p^{r-1}$-Borel cohomology spectral sequence 
for $THH(H\F_p)$ results from taking the part of \rref{ethh3} 
with non-negative powers
of $x$ (and all the differentials contained entirely in that part), while the 
$\Z/p^{r-1}$-Borel homology spectral sequence
for $THH(H\F_p)$ results from taking the part of \rref{ethh3} 
with negative powers of $x$, shifted by $-1$ (graded homologically).

One gets:
\beg{ethh7}{\begin{array}{l}\displaystyle
(E\Z/p^{r-1}_+\wedge THH(H\F_p))^{\Z/p^{r-1}}_{2i}=\Z/p^{\min(i+1,r)}\\[1ex]
\displaystyle
(E\Z/p^{r-1}_+\wedge THH(H\F_p))^{\Z/p^{r-1}}_{2i+1}=\Z/p^{\min(i+1,r-1)}
\end{array}
}
for $i\geq 0$ (it is $0$ for $i<0$).

The calculation of $THH(H\F_p)^{\Z/p^{r-1}}$ is then completed by induction: For any $\Z/p^{r-1}$-equivariant
spectrum $E$, we have a cofibration sequence
\beg{ethh8}{
(E\Z/p^{r-1}_+\wedge E)^{\Z/p^{r-1}}\r E^{\Z/p^{r-1}}\r (\Phi^{\Z/p}E)^{\Z/p^{r-2}}.
}
Since $THH(H\F_p)$ is a cyclotomic spectrum, its coefficients are the coefficients of the homotopy fiber
of the connecting map
\beg{ethh9}{THH(H\F_p)^{\Z/p^{r-2}}\r \Sigma (E\Z/p^{r-1}_+\wedge THH(H\F_p))^{\Z/p^{r-1}}.
}
We know the target by \rref{ethh7}. The induction gives
\beg{ethh10}{
THH(H\F_p)^{\Z/p^{r-1}}_*=\Z/p^r[y]
}
where $y$ has homological degree $2$. Assuming this inductively with $r$ replaced by $r-1$,
the connecting map \rref{ethh9} on coefficients
(which decreases homological degree by $1$) is onto in odd degrees, and the even terms have an 
extension, creating the answer \rref{ethh10} additively. The multiplicative answer then also
follows inductively from the fact that the second map \rref{ethh8}
is a ring map when $E$ is a ring spectrum.

We also recall the fact that the map $R$ (which is defined in \cite{hm}
by composing the second map \rref{ethh8}
with the cyclotomic structure map) sends $y$ to $py$, while the map $F$ (defined as the forgetful map)
sends $y$ to $y$. This implies that, letting $TR$ be the microscope of the map $R$, we have
\beg{ethh11}{
TR(H\F_p)=H\Z_p.
}
All the spectra discussed in the process of the calculation are module spectra over \rref{ethh11},
and thus are generalized Eilenberg-MacLane spectra.

\vspace{5mm}

\section{Calculation of $TR(THC(H\F_p))$} \label{str}

We now combine the material of the last two sections to calculate 
the coefficients of $THH(THC(H\F_p))^{\Z/p^{r-1}}$,
and thereby prove Theorem \ref{t2}.
First of all, we can write
\beg{estr0}{THC(H\F_p)=Cobar_{H\F_p}(H\F_p,H\F_p\wedge H\F_p,H\F_p),}
so the coefficients are indeed dual to \rref{ethh1}. Here we write
$$\begin{array}{l}Cobar_{H\F_p}(H\F_p,H\F_p\wedge H\F_p,H\F_p)=C_{H\F_p}(H\F_p\wedge H\F_p)=\\
F_{H\F_p\wedge H\F_p}(B_{H\F_p}(H\F_p\wedge H\F_p,H\F_p\wedge H\F_p,H\F_p),H\F_p).
\end{array}
$$
In fact, there is a Hopf algebra structure, which implies
that we can write
\beg{estr1}{THC(H\F_p)_*=\Gamma_{\F_p}(\rho)
}
where $\Gamma$ denotes the divided polynomial power algebra
and the homological degree of the element $\rho$ is $-2$. Now
by \rref{estr0}, $THC(H\F_p)$ is an $H\F_p$-algebra, and (implicitly
assuming cofibrant replacements in every term), we can therefore, non-equivariantly,
write
\beg{estr2}{\begin{array}{l}
THH(THC(H\F_p))=\\
B_{H\F_p}(THC(H\F_p), THC(H\F_p)\wedge THC(H\F_p),THC(H\F_p)),
\end{array}
}
with the subscript indicating that the bar construction is performed in the category of $H\F_p$-modules. 
(To clarify the
$H\F_p$-module structures on the right hand side of
\rref{estr2}, note that there is a canonical $H\F_p$-module structure on $THC(H\F_p)$. On the
term $THC(H\F_p)\wedge THC(H\F_p)$, we can use the $H\F_p$-module structure from either factor.
To fix ideas, we will use the first factor.)

This is just an example of a general fact that, for a morphisms of $E_\infty$-algebras $A\r B$ and an
associative algebra $C$ over $B$ with a left $C$-module $N$ and a right $C$-module $M$, assuming 
cofibrancy throughout, we have
$$B_A(M,C,N)\sim B_B(M,C,N).$$
This is simply due to the fact that both 
$$B_A(C,C,N),\; B_B(C,C,N)$$
are $C$-cofibrant models of the left $C$-module $N$, to which we are applying the 
functor $M\wedge_C?$. This principle is often used in spectral as well as classical algebra,
including the correct calculation of $THH(\F_p)$ via \rref{ethh1}.

This means 
we have again an Eilenberg-MacLane spectral sequence. In fact, we have a further filtration on \rref{estr2},
in the category of $H\F_p$-algebras, (coming from the augmentation ideal of the first factor of the 
middle term)
with associated graded object
\beg{estr3}{\begin{array}{l}
THC(H\F_p)\wedge_{H\F_p}B_{H\F_p}(H\F_p, H\F_p\wedge THC(H\F_p),H\F_p)=\\
THC(H\F_p)\wedge_{H\F_p}THH(H\F_p)\wedge_{H\F_p}B_{H\F_p}(C_{H\F_p}(H\F_p\wedge H\F_p))=\\
THC(H\F_p)\wedge_{H\F_p}THH(H\F_p)\wedge_{H\F_p}F(H\F_p,H\F_p)
\end{array}
}
(by Koszul duality). The coefficients of \rref{estr3} are
\beg{estr4}{
\F_p[\sigma]\otimes \Gamma_{\F_p}(\rho)\otimes A^*
}
where $A^*=F(H\Z/p,H\Z/p)_*$. On the other hand, by Theorem \ref{t1}, we have a map
\beg{estr5}{
THH(THC(H\F_p))\r F(THH(H\F_p),THH(H\F_p)).
}
(This is, in fact, even true $\Z/p^{r-1}$-equivariantly.) Non-equivariantly, however, all the elements
\rref{estr4} exist and are non-zero in the target of \rref{estr5}, and thus cannot support differentials, 
Therefore, we have proved that
\beg{estr6}{
THH(THC(H\F_p))_*=\F_p[\sigma]\otimes \Gamma_{\F_p}(\rho)\otimes A^*.
}
From this point on, the strategy for computing $THH(THC(H\F_p))^{\Z/p^{r-1}}_*$ mimics the strategy
for $THH(H\F_p)_*$, described in Section \ref{scalc1}. We begin by calculating 
the coefficients of the Borel homology spectrum
\beg{estr7}{
E\Z/p^{r-1}_+\wedge THH(THC(H\F_p))
}
via the Borel homology spectral sequence. In contrast with the Borel homology of $THH(H\F_p)$,
the differentials are somewhat different. First, we may filter the $E^2$-term
\beg{estr8}{
\F_p[\sigma]\otimes \Gamma_{\F_p}(\rho)\otimes A^*\{e_0,e_1,\dots\}
}
by powers of $\sigma$ (where $e_i$ is the generator of $H_i(\Z/p,\Z/p)$). This filtration is, in fact, also
realized on the spectral level. Now the associated graded object is a polynomial algebra in one
variable $\sigma$ over
\beg{estr9}{
\Gamma_{\F_p}(\rho)\otimes A^*\{e_0,e_1,\dots\}.
}
On \rref{estr9}, however, there is a $d^2$-differential given by
\beg{estr10}{
d^2(e_i)=Q_0^Re_{i-2}\overline{\sigma}
}
(where $\overline{\sigma}$ acts on $\Gamma_{\F_p}(\rho)$ as the dual variable $\sigma$ does
using the Hopf algebra structure on $THC(H\F_p)_*$, i.e by $\overline{\sigma}:\gamma_i(\rho)\mapsto
\gamma_{i-1}(\rho)$). This differential in fact comes from the extension
\rref{ethh5}:

\vspace{4mm}

\begin{sseqpage}[ymirror, no y ticks, y range = {0}{4}, x range = {0}{4},
  classes={draw = none}, xscale=2.2, y axis gap = 3em, x axis extend end = 2em]
  \begin{scope}[ background ]
\node at (-0.7,0) {0};
\foreach \n in {1,..., \ymax}{
\node at (-0.7,\n) {-\n};
}
\end{scope}
  \class["A^*e_0",red](0,0) \class["A^*e_1",red](1,0)
  \class["A^*e_2",red](2,0) \class["A^*e_3",red](3,0)
  \class["A^{*}e_4", red](4,0)
  \class["A^*\rho\cdot e_0",blue](0,2) \class["A^*\rho\cdot e_1",blue](1,2)
  \class["A^*\rho\cdot e_2"](2,2) \class["A^*\rho\cdot e_3"](3,2)
   \class["A^*\rho\cdot e_4"](4,2)
  \class["A^*\gamma_2\rho\cdot e_0",blue](0,4)
  \class["A^*\gamma_2\rho\cdot e_1",blue](1,4)
  \class["A^*\gamma_2\rho\cdot e_2"](2,4)
  \class["A^*\gamma_2\rho\cdot e_3"](3,4)
  \class["A^*\gamma_2\rho\cdot e_4"](4,4)
  \d["\cdot Q_0^R" { pos = 0.3}] 2(2,2)(0,0)
  \d["\cdot Q_0^R" { pos = 0.3}] 2(3,2)(1,0)
  \d["\cdot Q_0^R" { pos = 0.3}] 2(4,2)(2,0)
  \d["\cdot Q_0^R" { pos = 0.3}] 2(2,4)(0,2)
  \d["\cdot Q_0^R" { pos = 0.3}] 2(3,4)(1,2)
  \d["\cdot Q_0^R" { pos = 0.3}] 2(4,4)(2,2)
\end{sseqpage}

\vspace{4mm}

By the observations of Section \ref{scal}, however,
$Q_0^R$ is in fact exact on the Steenrod algebra. This means that the $E^3$-term of the slice \rref{estr9}
is a sum of
\beg{estr11}{
(\Gamma_{\F_p}(\rho))_{<0}\otimes A^*/Q_0^R\{e_0,e_1\}
}
where $(\Gamma_{\F_p}(\rho))_{<0}$ denotes the augmentation ideal of $\Gamma_{\F_p}(\rho)$
and
\beg{estr12}{
A^*/Q_0^R\{e_0,e_1,\dots\}.
}
When we put the slices together, we get a polynomial algebra on one generator $\sigma$ 
tensored with \rref{estr11} and \rref{estr12}. On 
\beg{estr20}{
A^*/Q_0^R\{e_0,e_1,\dots\}[\sigma],
}
we then get another component
of the $d^2$-differential 
\beg{estr21}{
d^2(e_i)=Q_0^Le_{i-2}\sigma.
}

\vspace{4mm}

\begin{center}

\begin{sseqpage}
  \class(0,0) \class(1,0) \class(2,0) \class(3,0) \class(4,0)
  \class["\sigma" above](0,2) \class(1,2) \class(2,2) \class(3,2) \class(4,2)
  \class(0,4) \class(1,4) \class(2,4) \class(3,4) \class(4,4)
  \d2(2,0)(0,2)
  \d2(3,0)(1,2)
  \d2(4,0)(2,2)
  \d2(2,2)(0,4)
  \d2(3,2)(1,4)
  \d["\cdot Q_0^L" { pos = 0.3, yshift = 5pt }]2(4,2)(2,4)
\end{sseqpage}

\end{center}

\vspace{4mm}

By the observations of Section \ref{scal}, $Q_0^L$ on $A^*/Q_0^R$ is a Koszul complex, with image $P_*$ and
cokernel $\Z/p\oplus P_*[1]$. Therefore, the part of the $E^3$-term on \rref{estr21} is a sum of
\beg{estr22}{
P_*\oplus P_*[1]\{e_0,e_1,\dots\}
}
and
\beg{estr23}{\Z/p[\sigma]\{e_0,e_1,\dots\}.
}
On \rref{estr23}, we then have the differential \rref{ethh4}. 

In summary, we obtain 

\begin{lemma}\label{lstr}
The coefficients of \rref{estr7} are a direct sum of \rref{estr11}, \rref{estr22}, and \rref{ethh7}.
\end{lemma}

\begin{proof}
The terms \rref{ethh7} come from \rref{estr23}. The other terms were already explained. The
terms \rref{ethh7} are the only ones which allow multiplicative extension, by degree restctions
and module structure over $THH_{\Z/p^{r-1}}(\F_p)$. The extensions in the terms \rref{ethh7}
were already disussed.
\end{proof}

Next, we use \rref{ethh8} to obtain

\begin{proposition}\label{pstr}
We have
\beg{estr30}{\begin{array}{l}
THH(THC(H\F_p))^{\Z/p^{r-1}}_*=
F(H\Z/p^r,H\Z/p^r)_*[y]\otimes \Gamma_{\Z/p^r}(\rho).
\end{array}
}
Additionally, the map $R$ of \cite{hm} sends $y$ to $py$ and $\gamma_i\rho$ to $\gamma_i\rho$.
\end{proposition}

\begin{proof}
Analogously to \rref{ethh9}, the coefficients of $THH(THC(H\F_p))^{\Z/p^{r-1}}$ are
the coefficients of the homotopy fiber of the morphism of spectra
\beg{estr31}{\diagram
THH(THC(H\F_p))^{\Z/p^{r-2}}\dto\\ \Sigma (E\Z/p^{r-1}_+\wedge THH(
THC(H\F_p)))^{\Z/p^{r-1}}.
\enddiagram}
The target is computed in Lemma \ref{lstr}. Thus, we proceed again by induction on 
$r$. Assuming the statement is true with $r$ replaced by $r-1$, we have an inductive calculation 
of the source of \rref{estr31}. Thus, we need to compute the connecting map. To this end, 
use the computations \rref{est2}, \rref{est3}.
Every copy of $P[1]$ in the source is isomorphically
mapped to a copy of $P$ in the target. At the augmentation ideal of $\Gamma(\rho)$, in fact, we
have a sum of copies of the exension of Lemma \ref{lst} with $k=r-1$, $\ell=1$. On the second component
$F(H\Z,H\Z/p^{r-1})_*[y][-1]$ from
\rref{est3} of the
\beg{estr32}{F(H\Z/p^{r-1},H\Z/p^{r-1})_*[y]} 
in the source of \rref{estr31}, we also have the standard 
Bockstein extension. On the $\Z/p^{r-1}$ part of the first component 
$F(H\Z,H\Z/p^{r-1})_*[y]$ from
\rref{est3} of \rref{estr32}, we have, in fact, the same map and extension 
as in \rref{ethh9}. Multiplicative considerations (which follow from the behavior of the unit) 
complete the proof.
\end{proof}

\begin{proof}[Proof of Theorem \ref{t2}]
Apply Proposition \ref{pstr}, and take the limit \rref{est10}. On the limit step, no $\lim^1$
occurs due to the Mittag-Leffler condition being in effect.
\end{proof}


\vspace{10mm}
\begin{thebibliography}{99}

\bibitem{and1} F. Andriopoulos: $TR$ and the $r$-Nygaard filtered prismatic cohomology, arXiv: 2412.00914

\bibitem{and2} F. Andriopoulos: $TR$ with logarithmic poles and the de Rham-Witt complex, arXiv: 2412.00929

\bibitem{tc1} V. Angeltveit, A. J. Blumberg, T. Gerhardt, M. A. Hill, T. Lawson, M. Mandell:
Topological cyclic homology via the norm,
{\em Doc. Math.} 23 (2018), 2101-2163

\bibitem{bh} T. Bachmann and M. Hoyois: Norms in motivic homotopy theory,
\emph{Ast\'{e}risque} no. 425, AMS 2021. 208 pp.

\bibitem{bms} B. Bhatt, M. Morrow, P. Scholze: Topological Hochschild homology and
integral $p$-adic Hodge theory, {\em Publ.math. IHES} 129 (2019), 199-310

\bibitem{bms1} B. Bhatt, M. Morrow, P. Scholze: Integral $p$-adic Hodge theory, 
{\em Publ.math. IHES} 128, (2018) 219-397

\bibitem{tc} M. B\"{o}kstedt, W. C. Hsiang, I. Madsen: The cyclotomic trace and algebraic K-theory of spaces,
{\em Invent. Math.} 111 (1993), no. 3, 465-539




\bibitem{bore}  A. M. Bohmann: A comparison of norm maps. With an appendix by Bohmann and Emily Riehl: {\em Proc. Amer. Math. Soc.} 142 (2014), no. 4, 1413-1423

\bibitem{caruso} J. Caruso: Operations in equivariant $\Z/p$-cohomology, 
{\em Math. Proc. Cambridge Philos. Soc.}126(1999), no.3, 521-541

\bibitem{dt} A. Dold, R. Thom:
Quasifaserungen und unendliche symmetrische Produkte, {\em Ann. of Math.} (2)   67 (1958), 239-281

\bibitem{dm}  B. I. Dundas, R. McCarthy: Topological Hochschild homology of ring functors and exact categories, 
{\em J. Pure Appl. Algebra} 109 (1996), no. 3, 231-294

\bibitem{ekmm} A. Elmendorf, I. Kriz, M. Mandell, P. May: 
{\em Rings, modules and algebras in stable homotopy theory}, 
AMS Mathematical Surveys and Monographs, Volume 47, 1997. \bibitem{gh} P.G.Goerss, M.J.Hopkins: Moduli spaces of commutative ring spectra, 
{\em Structured Ring Spectra} 151-200, London Math. Soc. Lecture Note Ser. 315, Cambridge Univ. Press, 2004

\bibitem{gk} V. Ginzburg, M. Kapranov: 
Koszul duality for operads, {\em Duke Math. J.} 76 (1994), no. 1, 203-272, erratum: 
{\em Duke Math. J.} 80 (1995), no. 1, 293




\bibitem{hm} L. Hesselholt, I. Madsen: On the K-theory of finite algebras over Witt
vectors of perfect fields,
{\em Topology} 36 (1997), 29-101

\bibitem{hm1} L. Hesselholt, I. Madsen: On the de Rham-Witt complex in mixed
characteristic,
{\em Ann. Sci. Ec. Norm. Sup.} 37 (4) (2004), 1-43.

\bibitem{hm2} L. Hesselholt, I. Madsen: On the K-theory of nilpotent endomorphisms,
{\em Homotopy methods in algebraic topology
(Boulder, CO, 1999)}, pp. 127-140, Contemp. Math., 271, Amer. Math. Soc.,
Providence, RI, 2001.

\bibitem{hhr} M. A. Hill, M. J. Hopkins, D. C. Ravenel:  On the nonexistence of elements of Kervaire invariant one,
{\em Ann. of Math.} (2) 184 (2016), no. 1, 1-262

\bibitem{hucacti} P. Hu: The Hochschild cohomology of a Poincar\'{e} algebra, arXiv:0707.4118

\bibitem{hu} P. Hu. Base change functors in the $\mathbb{A}^1$-stable homotopy category.
Equivariant stable homotopy theory and related areas (Stanford, CA, 2000).
\emph{Homology, Homotopy App.} 3 (2001), no. 2, 417-451

\bibitem{hksf}  P. Hu, I. Kriz, P. Somberg: On some adjunctions in equivariant stable homotopy theory, 
{\em Algebr. Geom. Topol.} 18 (2018), no. 4, 2419-2442

\bibitem{hksthh} P. Hu, I. Kriz, P. Somberg: On the equivariant motivic filtration of the
topological Hochschild homology of polynomial algebras,
{\em Advances in Mathematics}, Volume 412, 108803, 2023

\bibitem{drt} P.Hu, I.Kriz, P.Somberg: 
Derived representation theory of Lie algebras and stable homotopy categorification 
of $sl_k$, {\em Adv. Math. 341} (2019) 367-439

\bibitem{hkszsteen} P. Hu, I. Kriz, P. Somberg, F. Zou: The $\Z/p$-equivariant dual Steenrod algebra
for an odd prime $p$, arXiv:2205.13427



\bibitem{hk} P. Hu, I. Kriz: Real-oriented homotopy theory nd an analogue of the
Adams-Novikov spectral sequence, {\em Topology} 40 (2001), no. 2, 317-399


\bibitem{kauf}  R. M. Kaufmann: A proof of a cyclic version of Deligne's conjecture via cacti, {\em Math. Res. Lett.} 15 (2008), no. 5, 901-921

\bibitem{koch} S. Kochman: Integral cohomology operations. {\em Current trends
in algebraic topology, Part 1 (London, Ont., 1981)}, 
CMS Conf. Proc., 2, Amer. Math. Soc., Providence, R.I., 1982, pp. 437-478

\bibitem{lms} L. G. Lewis, Jr., J. P. May, M. Steinberger, and J. E. McClure: 
{\em Equivariant stable homotopy theory}, Vol 1213 of Lecture Notes in Mathematics,
Springer-Verlag, Berlin, 1986

\bibitem{c2} J. Lurie: {\em Higher Algebra}, Harvard University, 2017

\bibitem{mcclures} J. McClure, J. Smith:
{\em A solution of Deligne's Hochschild cohomology conjecture}, 
Recent progress in homotopy theory (Baltimore, MD, 2000), 153-193,
Contemp. Math., 293
American Mathematical Society, Providence, RI, 2002

\bibitem{milnor} J. Milnor:
The Steenrod algebra and its dual,
{\em Ann. of Math.} (2) 67 (1958), 150-171

\bibitem{ns} T. Nikolaus, P. Scholze: On topological cyclic homology, {\em Acta Math.} 221 (2018), no. 2, 203-409,
Correction: {\em Acta Math.} 222 (2019), no. 1, 215-218

\bibitem{sw} K. Sankar, D. Wilson: On the $C_p$-equivariant dual Steenrod algebra,
{\em Proc. Amer. Math. Soc.} 150 (2022), no.8, 3635-3647

\bibitem{schecht} V. Schechtman:
Sur les alg\`{e}bres vertex attach\'{e}es aux vari\'{e}t\'{e}s alg\'{e}briques, {\em
Proceedings of the International Congress of Mathematicians, Vol. II} (Beijing, 2002), 525-532

\bibitem{voronov} A. A. Voronov:
Notes on universal algebra.(English summary)Graphs and patterns in mathematics and theoretical physics, 81-103,
{\em Proc. Sympos. Pure Math.,} 73
American Mathematical Society, Providence, RI, 2005

\end{thebibliography}
\end{document}